\documentclass[12 pt]{article}
\usepackage{amssymb}
\usepackage{amsthm}
\usepackage{amsmath}
\usepackage{mathrsfs}
\usepackage{verbatim}
\begin{document}
\title{Regularity of optimal transportation between spaces with different dimensions\footnote{The author was supported in part by an NSERC postgraduate scholarship.  This work was completed in partial fulfillment of the requirements of a doctoral degree in mathematics at the University of Toronto.}} \author{BRENDAN PASS \footnote{Department of Mathematics, University of Toronto, Toronto, Ontario, Canada, M5S 2E3 bpass@math.utoronto.ca.}}
\maketitle

\begin{abstract} We study the regularity of solutions to an optimal transportation problem where the dimension of the source is larger than that of the target.  We demonstrate that if the target is $c$-convex, then the source has a canonical foliation whose co-dimension is equal to the dimension of the target and the problem reduces to an optimal transportation problem between spaces with equal dimensions.  If the $c$-convexity condition fails, we do not expect regularity for arbitrary smooth marginals, but, in the case where the source is 2-dimensional and the target is 1 dimensional, we identify sufficient conditions on the marginals and cost to ensure that the optimal map is continuous.
\end{abstract}

\section{Introduction}
vvLet $X$ and $Y$ be smooth manifolds of dimensions $m$ and $n$, endowed with Borel probability measures $\mu$ and $\nu$, respectively. We say that a Borel map $F: X \rightarrow Y$ pushes $\mu$ forward to $\nu$ if for all Borel sets $A \subseteq Y$ we have $\nu(A)=\mu(F^{-1}(A))$.  For a given cost function $c: X \times Y \rightarrow \mathbb{R}$, Monge's optimal transportation problem is then to find the Borel map $F$ pushing $\mu$ forward to $\nu$ that minimizes the total transportation cost:
\begin{equation} \label{monge}
\int_{X}c(x,F(x))d\mu
\end{equation}
This can be viewed as a stricter version of the Kantorovich optimal transportation problem, which is to minimize:
\begin{equation*}
\int_{X \times Y}c(x,y)d\gamma
\end{equation*}
among all Borel probability measures $\gamma$ on $X \times Y$ such that the projections of $X \times Y$ onto $X$ and $Y$ push $\gamma$ forward to $\mu$ and $\nu$, respectively.  In fact, the usual method for finding solutions to Monge's problem is to first find the Kantorovich solution; one can then show that, under certain conditions, the solution $\gamma$ is concentrated on the graph of a function $F: X \rightarrow Y$ \cite{lev}\cite{g}\cite{bren}\cite{gm}\cite{Caf}.  These conditions cannot generally hold if $m<n$; in this case, however, there are known conditions under which $\gamma$ will concentrate on the graph of a function $H: Y \rightarrow X$ and so it is preferable to reformulate Monge's problem in terms of maps from $Y$ to $X$.  When minimizing (\ref{monge}), then, it is natural to restrict our attention to the case when $m\geq n$.

Monge's problem has numerous applications and has received a lot of attention from many different authors.  Questions about the existence and uniqueness of optimal maps have been resolved for a wide class of cost functions; much of the present research in optimal transportation aims to understand the structure of these optimizers.  A great deal of progress has been made in this direction, but it has mostly been restricted to the case when $m=n$; problems where $m>n$, on the other hand, have received very little attention.  Aside from being a natural mathematical generalization of the relatively well understood $m=n$ case, however, optimal transportation problems where $m$ and $n$ fail to coincide may have important applications; for example, in economics, optimal transportation type problems arise frequently and there is often no compelling reason to assume that $m=n$.  For a treatment of a related problem in an economic context, see \cite{bas} and \cite{ds}; the connections between these results and the present work will be explored by the present author in a separate paper.

In the $m=n$ case, understanding the regularity, or smoothness, of the optimal map, has grown into an active and exciting area of research in the past few years, due to a major breakthrough by Ma, Trudinger and Wang \cite{mtw}.  They identified a fourth order differential condition on $c$ (called \textbf{(A3S)} in the literature) which implies the smoothness of the optimizer, provided the marginals $\mu$ and $\nu$ are smooth.  Subsequent investigations by Trudinger and Wang \cite{tw1,tw2} revealed that these results actually hold under a slight weakening of this condition, called \textbf{(A3W)}, encompassing earlier results of Caffarelli \cite{c1,c2,c3}, Urbas \cite{u} and Delanoe \cite{d1,d2} when $c$ is the distance squared on either $\mathbb{R}^{n}$ on certain Riemannian manifolds and Wang for another special cost function \cite{wang}.  Loeper \cite{loeper} then verified that \textbf{(A3W)} is in fact necessary for the solution to be continuous for arbitrary smooth marginals $\mu$ and $\nu$.  Loeper also proved that, under \textbf{(A3S)}, the optimizer is Holder continuous even for rougher marginals; this result was subsequently improved by Liu \cite{liu}, who found a sharp Holder exponent.  Since then, many interesting results about the regularity of optimal transportation have been established \cite{km}\cite{km2}\cite{loeper2}\cite{lv}\cite{fr}\cite{frv}\cite{frv2}\cite{fkm}\cite{fkm2}.

This article focuses on adapting these results to the $m>n$ setting.  A serious obstacle arises immediately; the regularity theory of Ma, Trudinger, and Wang requires invertibility of the matrix of mixed second order partials $(\frac{\partial^2 c}{\partial x^i \partial y^j})_{ij}$, and its inverse appears explicitly in their formulations of $\textbf{(A3W)}$ and $\textbf{(A3S)}$.  When $m$ and $n$ fail to coincide, however, $(\frac{\partial^2 c}{\partial x^i \partial y^j})_{ij}$ clearly cannot be invertible.  Alternate formulations of the $\textbf{(A3W)}$ and $\textbf{(A3S)}$ that do not explicitly use this invertibility are known; however, they rely instead on local surjectivity of the map $y \mapsto D_xc(x,y)$, which cannot hold in our setting either.

Nonetheless, there is a certain class of costs for which our problem can easily be solved using the results from the equal dimensional setting.  Suppose 
\begin{equation}\label{quot}
c(x,y)=b(Q(x),y),
\end{equation} 
where $Q:X \rightarrow Z$ is smooth and $Z$ is a smooth manifold of dimension $n$.  In this case, it is not hard to show that the optimal map takes every point in each level set of $Q$ to a common $y$ and studying its regularity amounts to studying an optimal transportation problem on the $n$-dimensional spaces $Z$ and $Y$.  We will show that costs of this form are essentially the only costs on $X \times Y$ for which we can hope for regularity results for arbitrary smooth marginals $\mu$ and $\nu$.  Indeed, for the quadratic cost on Euclidean domains, the regularity theory of Caffarelli requires convexity of the target $Y$ \cite{c1}\cite{c2} and, for general costs, it became apparent in the work of Ma, Trudinger and Wang \cite{mtw} that continuity of the optimizer cannot hold for arbitrary smooth marginals unless $Y$ satisfies an appropriate, generalized notion of convexity.  Due to its dependence on the cost function, this condition is referred to as $c$-convexity;  when $m >n$, we will show that $c$-convexity necessarily fails \textit{unless} the cost function is of the form alluded to above.

In the next section, we will introduce preliminary concepts from the regularity theory of optimal transportation, suitably adapted for general values of $m \geq n$.  In the third section, we prove that $c$-convexity implies the existence of a quotient map $Q$ as discussed above.  We then show that the properties on $Z$ which are necessary for the optimal map to be continuous follow from analogous properties on $X$.

Given the preceding discussion, it is apparent that for cost functions that are not of the special form (\ref{quot}), there are smooth marginals for which the optimal map is discontinuous.  However, as the condition (\ref{quot}) is so restrictive, it is natural to ask about regularity for costs which are not of this form; any result in this direction will require stronger conditions on the marginals than smoothness.  In the final section of our paper, we address this problem when $m=2$ and $n=1$.

\textbf{Acknowledgment:}  The author is pleased to thank Robert McCann and Paul Lee for fruitful discussions during the course of this work.

\section{Conditions and definitions}
Here we develop several definitions and conditions which we will require in the following sections.  We begin with some basic notation.  In what follows, we will assume that $X$ and $Y$ may be smoothly embedded in larger manifolds, in which their closures, $\overline{X}$ and $\overline{Y}$, are compact.  If $c$ is differentiable, we will denote by $D_xc(x,y)$ its differential with respect to $x$.  If $c$ is twice differentiable, $D^2_{xy}c(x,y)$ will denote the map from the tangent space of $Y$ at $y$, $T_yY$, to the cotangent space of $X$ at $x$, $T^*_xX$, defined in local coordinates by
\begin{equation*}
 \frac{\partial}{\partial y^i} \mapsto \frac{\partial ^2 c(x,y)}{\partial y^i \partial x^j}dx^j
\end{equation*}
where summation on $j$ is implicit, in accordance with the Einstein summation convention.  $D_yc(x,y)$ and $D^2_{yx}c(x,y)$ are defined analogously.

A function $u:X \rightarrow \mathbb{R}^n$ is called $c$-concave if $u(x)=\inf_{y \in Y} c(x, y)-u^c(y)$, where $u^c(y):=\inf_{x \in X} c(x, y)-u(x)$.

Next, we introduce the concept of $c$-convexity, which first appeared in Ma, Trudinger and Wang.  
\newtheorem{con}{Definition}[section]
\begin{con}
We say domain $Y$ looks $c$-convex from $x \in X$ if $D_xc(x,Y)=\{D^2_{x}c(x,y) | y \in Y\}$ is a convex subset of $T_xX$.  We say $Y$ is $c$-convex with respect to $X$ if it looks $c$-convex from every $x \in X$.
\end{con}
Our next definition is novel, as it is completely irrelevant when $m=n$.  It will, however, play a vital role in the present setting.
\newtheorem{lin}[con]{Definition}
\begin{lin}
We say domain $Y$ looks $c$-linear from $x \in X$ if $D_xc(x,Y)$ is contained in an $n$-dimensional, linear subspace of $T_xX$.  We say $Y$ is $c$-linear with respect to $X$ if it looks $c$-linear from every $x \in X$.
\end{lin}

When $m=n$, $c$-linearity is automatically satisfied.  When $m>n$, this is no longer true, although $c$-convexity clearly implies $c$-linearity.  

We will also have reason to consider the level set of $\overline{x} \mapsto D_yc(\overline{x},y)$ passing through $x$, $L_x(y):=\{\overline{x} \in X: D_yc(\overline{x},y)=D_yc(x,y)\}$.  

Let us now state the first three regularity conditions introduced by Ma, Trudinger and Wang.\\
\\
\textbf{(A0)}: The function $c \in C^4(\overline{X} \times \overline{Y})$.\\
\textbf{(A1)}: (Twist) For all $x \in X$, the map $y \mapsto D_xc(x,y)$ is injective on $\overline{Y}$.\\
\textbf{(A2)}: (Non-degeneracy) For all $x \in X$ and $y \in Y$, the map $D^{2}_{xy}c(x,y):T_yY \rightarrow T^{*}_xX$ is injective.\\
\newtheorem{rem}[con]{Remark}
\begin{rem}
When $m=n$, a bi-twist hypothesis is required to prove regularity of the optimal map; in addition to \textbf{(A1)}, one must assume $x \mapsto D_yc(x,y)$ is injective on $X$ for all $y \in Y$.  Clearly, such a condition cannot hold if $m>n$; in fact, the non-degeneracy condition and the implicit function theorem imply that the level sets $L_x(y)$ of this mapping are smooth $m-n$ dimensional hypersurfaces.  Later, we will assume that the these level sets are connected.  When $m=n$, non-degeneracy implies that each $L_x(y)$ consists of finitely many isolated points, in which case connectedness implies that it is in fact a singleton, or, equivalently, that $x \mapsto D_yc(x,y)$ is injective.
\end{rem}
The statements of \textbf{(A3W)} and  \textbf{(A3S)}, the most important regularity conditions, require a little more machinery.  For a twisted cost, the mapping $y \mapsto D_xc(x,y)$ is invertible on its range.  We define the $c$-exponential map at $x$, denoted by $c\text{-}exp_x(\cdot)$, to be its inverse; that is, $D_xc(x,c\text{-}exp_x(p))=p$ for all $p \in D_xc(x,Y)$.

\newtheorem{curv}[con]{Definition}
\begin{curv}
 Let $x \in X$ and $y \in Y$.  Choose tangent vectors $\textbf{u}\in T_x X$ and $\textbf{v} \in T_yY$.  Set $\textbf{p}=D_xc(x,y) \in T_x^{*}X$ and $\textbf{q} = (D^2_{xy}c(x,y))\cdot \textbf{v} \in T_x^{*}X$; note that if $Y$ looks $c$-linear at $x$, $\textbf{p}+t\textbf{q} \in D_xc(x,Y)$ for small $t$.  For any smooth curve $\beta(s)$ in $X$ with $\beta(0)=x$ and $\frac{d\beta}{ds}(0)=\textbf{u}$, we define the Ma, Trudinger Wang curvature at $x$ and $y$ in the directions $\textbf{u}$ and $\textbf{v}$ by:
\begin{equation*}
 MTW_{xy}\langle \textbf{u},\textbf{v}\rangle : =-\frac{3}{2}\frac{\partial^4c}{\partial s^2\partial t^2}c(\beta(s),c\text{-}exp_x(\textbf{p}+t\textbf{q}))
\end{equation*}

\end{curv}

We are now ready to state the final conditions of Ma, Trudinger and Wang.  Because they are designed to deal with the general case $ m \geq n$, our formulations look somewhat different from those found in \cite{mtw}; when $m=n$, they reduce to the standard conditions.\\
\\
\textbf{(A3W)}: For all $x \in X$, $y \in Y$, $\textbf{u}\in T_x X$ and $\textbf{v} \in T_yY$ such that $\textbf{u} \cdot D^2_{xy}c(x,y) \cdot \textbf{v}=0$, $MTW_{xy}\langle \textbf{u},\textbf{v}\rangle \geq 0$.
\\
\textbf{(A3S)}: For all $x \in X$, $y \in Y$, $\textbf{u}\in T_x X$ and $\textbf{v} \in T_yY$ such that $\textbf{u} \cdot (D^2_{xy}c(x,y)) \cdot \textbf{v}=0$,  $\textbf{u} \cdot (D^2_{xy}c(x,y)) \neq 0$ and $\textbf{v} \neq 0$ we have  $MTW_{xy}\langle \textbf{u},\textbf{v}\rangle > 0$. \\

If $m=n$, non-degeneracy implies that the condition $\textbf{u}\cdot (D^2_{xy}c(x,y)) \neq 0$ is equivalent to $\textbf{u} \neq 0$.

\section{Regularity of optimal maps}
The following theorem asserts the existence of an optimal map.  It is due to Levin \cite{lev} in the case where $X$ is a bounded domain in $\mathbb{R}^m$ and $\mu$ is absolutely continuous with respect to Lebesgue measure.  The following version can be proved in the same way; see also Brenier \cite{bren}, Gangbo \cite{g}, Gangbo and McCann \cite{gm} and Caffarelli \cite{Caf}.

\newtheorem{exist}{Theorem}[section]

\begin{exist}
 Suppose $c$ is twisted and $\mu(A)=0$ for all Borel sets $A \subseteq X$ of Hausdorff dimension less than or equal to $m-1$.  Then the Monge problem admits a unique solution $F$ of the form $F(x) =c$-exp$(x,Du(x))$ for some $c$-concave function $u$. 
\end{exist}

The following example confirms the necessity of $c$-convexity to regularity.  It is due to Ma, Trudinger and Wang in the case where $m=n$; their proof applies to the $m \geq n$ case as well.

\newtheorem{connec}[exist]{Theorem}
\begin{connec}
 Suppose there exists some $x \in X$ such that $Y$ does not look $c$-convex from $x$.  Then there exist smooth measures $\mu$ and $\nu$ for which the optimal map is discontinuous.
\end{connec}

As $c$-convexity implies $c$-linearity, this example verifies that we cannot hope to develop a regularity theory in the absence of $c$-linearity.  The following lemma demonstrates that, under the $c$-linearity hypothesis, the level sets $L_x(y)$ are the same for each $y$, yielding a canonical foliation of the space $X$.

\newtheorem{linlev}[exist]{Lemma}
\begin{linlev}(i) $Y$ looks $c$-linear from $x \in X$ if and only if $T_x(L_x(y))$ is independent of $y$; that is $T_x(L_x(y_0))=T_x(L_x(y_1))$ for all $y_0,y_1 \in Y$.
\\(ii) If the level sets $L_x(y)$ are all connected, then $Y$ is $c$-linear with respect to $X$ if and only if $L_x(y)$ is independent of $y$ for all $x$
\end{linlev}

\begin{proof}
We first prove (i).  The tangent space to $L_x(y)$ at $x$ is the null space of the map $D^2_{yx}c(x,y):T_xX \mapsto T^*_yY$, which, in turn, is the orthogonal complement of the range of $D^2_{xy}c(x,y):T_yY \mapsto T^*_xX$.  Therefore, $T_x(L_x(y))$ is independent of $y$ if and only if the range of $D^2_{xy}c(x,y)$ is independent of $y$.   But $D^2_{xy}c(x,y)$ is the differential of the map $y \mapsto D_xc(x,y)$ (making the obvious identification between $T_x^*X$ and its tangent space at a point) and so its range is independent of $y$ if and only if the image of this map is linear.

To see (ii), note that (i) implies $Y$ is $c$-linear with respect to $X$ if and only if $T_x(L_x(y_0))=T_x(L_x(y_1))$ for all $x \in X$  and all $y_0,y_1 \in Y$.  But $T_x(L_x(y_0))=T_x(L_x(y_1))$ for all $x$ is equivalent to $L_x(y_0)=L_x(y_1)$ for all $x$; this immediately yields (ii).
\end{proof}

For the remainder of this section, we will assume that $L_x(y)$ is connected and independent of $y$ for all $x$ and we will denote it simply by $L_x$.  In this case, we will demonstrate now that points in the same level set are indistinguishable from an optimal transportation perspective.  The $L_x$'s define a canonical foliation of $X$ and our problem will be reduced to an optimal transportation problem between $Y$ and the space of leaves of this foliation.  More precisely, we define an equivalence relation on $X$ by $x \sim \overline{x}$ if $\overline{x} \in L_x$.  We then define the quotient space $Z=X/\sim$ and the quotient map $Q: X \rightarrow Z$.  Note that, for any fixed $y_0 \in Y$, the map $x \mapsto D_yc(x,y_0) \in T_{y_0}Y$ has the same level sets as $Q$ (namely the $L_x$'s) and is smooth by assumption.  Furthermore, the non-degeneracy condition implies that this map is open and hence a quotient map.  We can therefore identify $Z \approx D_yc(X,y_0)$ with a subset of the cotangent space $T_{y_0}^*Y$.  In particular, $Z$ has a smooth structure, and, if $c$ satisfies \textbf{(A0)}, $Q$ is $C^3$.

Our strategy now will be to show that if $F: X \rightarrow Y$ is the optimal map, then $F$ factors through $Q$; $F=T \circ Q$. As $Q$ is smooth, this will imply that treating the smoothness of $F$ reduces to studying the smoothness of $T$.  To this end, we will show that $T$ itself solves an optimal transportation problem with marginals $\alpha=Q_{\#} \mu$ on $Z$ and $\nu$ on $Y$ relative to the cost function $b(z,y)$ defined uniquely by: 
\begin{eqnarray*}
&D_yb(z,y)=D_yc(x,y), \text{ for } x \in Q^{-1}(z)\\
&b(z,y_0)=0
\end{eqnarray*}
As $Z$ and $Y$ share the same dimension, the regularity theory of Ma, Trudinger and Wang will apply in this context. 

We first obtain a useful formula for the cost function $b$.
\newtheorem{formula}[exist]{Proposition}
\begin{formula}\label{formula}
 For any $z \in Z$, $y\in Y$ and $x \in Q^{-1}(z)$, we have $b(z,y)=c(x,y)-c(x,y_0)$.
\end{formula}
\begin{proof}
For $y=y_0$ the result follows immediately from the definition of $h$.  As $D_yb(z,y)=D_yc(x,y)$ for all $y$, the formula holds everywhere.
\end{proof}

Note that this implies $c(x,y)=b(Q(x),y)+c(x,y_0)$, which is equivalent to $b(Q(x),y)$ for optimal transportation purposes.

\newtheorem{ccon}[exist]{Lemma}
\begin{ccon} \label{con}
For any $x_0, x_1 \in L_x$, $\overline{y} \in Y$ and $c$-concave $u$ we have $u(x_0)=c(x_0, \overline{y})-u^{c}(\overline{y})$ if and only if $u(x_1)=c(x_1, \overline{y})-u^{c}(\overline{y})$. 
\end{ccon}

\begin{proof}
 First note that as $D_yc(x_0,y)-D_yc(x_1,y)=0$ for all $y \in Y$, the difference $c(x_0,y)-c(x_1,y)$ is independent of $y$.  Now, suppose $u(x_0)=c(x_0, \overline{y})-u^{c}(\overline{y})$.  Then 
\begin{eqnarray*}
 u(x_1)&=&\inf_{y \in Y} c(x_1, y)-u^c(y) \\
&=&\inf_{y \in Y} \big( c(x_1, y)-c(x_0,y)+c(x_0,y)-u^c(y) \big)\\
&=& c(x_1, \overline{y})-c(x_0,\overline{y})+\inf_{y \in Y} \big( c(x_0,y)-u^c(y) \big)\\
&=& c(x_1, \overline{y})-c(x_0,\overline{y})+u(x_0) \\ 
&=& c(x_1, \overline{y})-u^c(\overline{y}) \\ 
\end{eqnarray*}
The proof of the converse is identical.
\end{proof}

\newtheorem{fac}[exist]{Proposition}
\begin{fac}
 Suppose $c$ is twisted and $\mu$ doesn't charge sets of Hausdorff dimension $m-1$.  Let $F: X \rightarrow Y$ be the optimal map.  Then there exists a map $T: Z \rightarrow Y$ such that $F=T \circ Q$, $\mu$ almost everywhere.  Moreover, $T$ solves the optimal transportation problem on $Z \times Y$ with cost function $b$ and marginals $\alpha$ and $\nu$.
\end{fac}

\begin{proof}
 It is well known that there exists a $c$-concave functions $u(x)$ such that, for $\mu$ almost every $x$, there is a unique $y \in Y$ such that $u(x)=c(x,y)-u^{c}(y)$; in this case, $F(x)=y$.

For $\alpha$ almost every $z \in Z$, Lemma \ref{con} now implies that there is a unique $y \in Y$ such that $u(x)=c(x,y)-u^{c}(y)$ for all $x \in Q^{-1}(z)$; define $T(z)$ to be this $y$.  In then follows immediately that $F= T \circ Q$, $\mu$ almost everywhere, and that $T$ pushes $\alpha$ to $\nu$.

Now, suppose $G:Z \rightarrow Y$ is another map pushing $\alpha$ to $\nu$.  Then $G \circ Q$ pushes $\mu$ to $\nu$ and because of the optimality of $F = Q \circ T$ we have 
\begin{equation}
\int_{X}c(x,T \circ Q(x))d\mu \leq \int_{X}c(x,G \circ Q(x))d\mu.\label{prim}
\end{equation}
Now, using Proposition \ref{formula} we have
\begin{eqnarray*}
 \int_{X}c(x,T \circ Q(x))d\mu &=&\int_{X}b(Q(x),T \circ Q(x))+c(x,y_0)d\mu \\
&=&\int_{Z}b(z,T(z))d\alpha+\int_{X}c(x,y_0)d\mu
\end{eqnarray*}
Similarly, 
\begin{equation*}
 \int_{X}c(x,G \circ Q(x))d\mu = \int_{Z}b(z,G(z))d\alpha+\int_{X}c(x,y_0)d\mu
\end{equation*}
and so (\ref{prim}) becomes
\begin{equation*}
\int_{Z}b(z,T(z))d\alpha \leq \int_{Z}c(z,G(z))d\alpha \\
\end{equation*}
Hence, $T$ is optimal.
\end{proof}

Having established that the optimal map $F$ from $X$ to $Y$ factors through $Z$ via the quotient $Q$ and the optimal map $T$ from $Z$ to $Y$, we will now study how the regularity conditions \textbf{(A1)-(A3S}) for $c$ translate to $b$.

Proposition \ref{formula} also allows us to understand the derivatives of $b$ with respect to $z$.  Pick a point $z_0 \in Z$ and select $x_0 \in Q^{-1}(z_0)$.  Now, let $S$ be an $m$-dimensional surface passing though $x_0$ which intersects $L_{x_0}$ transversely.  As the null space of the map $D^2_{yx}c(x,y_0): T_xX \rightarrow T_y^*Y$ is precisely  $T_xL_x$ for any $y$, it is invertible when restricted to $T_xS$; by the inverse function theorem, the map $D_yc(\cdot,y_0)$ restricts to a local diffeomorphism on $S$.  For all $z$ near $z_0$, there is a unique $x \in S \cap Q^{-1}(z)$ and we have $b(z,y)=c(x,y)-c(x,y_0)$; we can now identify $D_zb(z,y) \approx D_xc|_{S \times Y}(x,y)-D_xc|_{S \times Y}(x,y_0)$ and $D_{zy}^2b(z,y) \approx D_{xy}^2c|_{S \times Y}(x,y)$.  We use this observation to prove the following result.

\newtheorem{prop}[exist]{Theorem}
\begin{prop} (i) If $c$ is twisted, $b$ is bi-twisted.\\
(ii) If $c$ is non-degenerate, $b$ is non-degenerate.\\
(iii)If $Y$ is $c$-convex, it is also $b$-convex.  
\end{prop}
\begin{proof}
 The injectivity of $z \mapsto D_yb(z,y)$ follows immediately from the the definition of $b$.  Injectivity of $y \mapsto D_zb(z,y)$ and non-degeneracy follow from the preceding identification.  

Note that transversality implies $T_x^*X =T_x^*L_x \oplus T_x^*S$.  Our local identification between $Z$ and $S$ identifies the projection of the range $D_xc(x,Y)$ onto $T_x^*S$ with $D_zb(z,Y)$.  As the projection of a convex set is convex, the $b$-convexity of $Y$ now follows from its $c$-convexity.

\end{proof}

\newtheorem{a3w}[exist]{Theorem}
\begin{a3w}
 The following are equivalent:
\begin{enumerate}
 \item $b$ satisfies \textbf{(A3W)}.
 \item $c$ satisfies \textbf{(A3W)}.
 \item $c$ satisfies \textbf{(A3W)} when restricted to any smooth surface $S \subseteq X$ of dimension $m$ which is transverse to each $L_x$ that it intersects.
\end{enumerate}

\end{a3w}

\begin{proof} The equivalence of (1) and (3) follow immediately from our identification.  Clearly, (2) implies (3); to see that (3) implies (2) it suffices to show $MTW_{xy}\langle \textbf{u},\textbf{v}\rangle = 0$ when $\textbf{u} \in T_xL_x $, as $MTW_{xy}$ is linear in $\textbf{u}$.  Choosing a curve $\beta(s) \in L_x$ such that $\beta(0)=x$ and $\frac{d\beta}{ds}(0)=\textbf{u}$ and $\textbf{p},\textbf{q}$ as in the definition, we have 
\begin{equation*}
\frac{d\beta}{ds}(s) \in T_{\beta(s)}L_{\beta(s)}=\text{null}\big(D^2_{xy}c(\beta(s),c\text{-}exp_x(\textbf{p}+t\textbf{q}))\big).
\end{equation*}   
for all $s$ and $t$, yielding
\begin{equation*} 
  \frac{d^2}{dsdt}c(\beta(s),c\text{-}exp_x(\textbf{p}+t\textbf{q}))=\frac{d\beta}{ds}\cdot D^2_{xy}c\big(\beta(s),c\text{-}exp_x(\textbf{p}+t\textbf{q})\big) \cdot \frac{d(c\text{-}exp(\textbf{p}+t\textbf{q}))}{dt}=0
 \end{equation*}
Hence, $MTW_{xy}\langle \textbf{u},\textbf{v}\rangle = 0$
\end{proof}

\newtheorem{a3s}[exist]{Theorem}
\begin{a3s}
 The following are equivalent:
\begin{enumerate}
 \item $b$ satisfies \textbf{(A3S)}.
 \item $c$ satisfies \textbf{(A3S)}.
 \item $c$ satisfies \textbf{(A3S)} when restricted to any smooth surface $S \subseteq X$ of dimension $m$ which is transverse to each $L_x$ that it intersects.
\end{enumerate}
\end{a3s}
\begin{proof}
The equivalence follows immediately from the identification, after observing that the $v\cdot(D^2_{xy}c(x,y)) \neq 0$ condition in the definition of \textbf{(A3S)} excludes the non-transverse directions.
\end{proof}

Various regularity results for $T$ (and therefore $F$) now follow from the regularity results of Ma, Trudinger and Wang \cite{mtw}, Loeper \cite{loeper} and Liu \cite{liu}.  Note, however, that these results all require certain regularity hypotheses on the marginals; to apply them in the present context, we must check these conditions on $\alpha$, rather than $\mu$.  A brief discussion on whether the relevant regularity conditions on $\mu$ translate to $\alpha$ therefore seems in order.

First, suppose $X$ is a bounded domain in $\mathbb{R}^{n}$ and $\mu =f(x)dx$ is absolutely continuous with respect to $m$-dimensional Lebesgue measure.  Then $\alpha$ is absolutely continuous with respect to $n$-dimensional
Lebesgue measure with density $h(z)$ given by the coarea formula:
\begin{equation*}\label{coarea}
h(z):=\int_{Q^{-1}(z)}\frac{f(x)}{JQ(x)}dH^{m-n}(x)
\end{equation*}
where $JQ$ is the Jacobian of the map $Q$, restricted to the orthogonal complement of $T_xL_x$.
\newtheorem{lp}[exist]{Lemma}
\begin{lp}
 Suppose $f \in L^p(X)$ (with respect to Lebesgue measure on $X$) for some $p \in [1,\infty]$.  Then $h \in L^p(Z)$.
\end{lp}

\begin{proof}
 We have $h^p(z)=(\int_{Q^{-1}(z)}\frac{f(x)}{JQ(x)}dH^{m-n}(x))^p$.  Normalizing and applying Jensen's inequality yields:
\begin{eqnarray*}
 \frac{h^p(z)}{C^p(z)} &\leq& \int_{Q^{-1}(z)}\frac{f^p(x)}{(JQ(x))^pC(z)}dH^{m-n}(x) \\
& \leq & \int_{Q^{-1}(z)}\frac{f^p(x)}{JQ(x)C(z)K^{p-1}}dH^{m-n}(x) 
\end{eqnarray*}
where $C(z)$ is the $(m-n)$-dimensional Hausdorff measure of $Q^{-1}(z)$ and $K>0$ is a global lower bound on $JQ(x)$.  Letting $C$ be a global upper bound on $C(z)$ and integrating over $z$ implies:

\begin{eqnarray*}
 \int h^p(z)dz & \leq & \int \int_{Q^{-1}(z)}\frac{f^p(x)C^{p-1}(z)}{JQ(x)^pK^{p-1}}dH^{m-n}(x) dz \\
& \leq & \frac{C^{p-1}}{K^{p-1}}\int \int_{Q^{-1}(z)}\frac{f^p(x)}{JQ(x)^p}dH^{m-n}(x) dz \\
&=& \frac{C^{p-1}}{K^{p-1}}\int f^p(x) dx < \infty
\end{eqnarray*}
where we have again used the coarea formula in the last step.
\end{proof}
Let us note, however, that an analogous result does not hold for the weaker condition introduced by Loeper \cite{loeper}, which requires that for all $x \in X$ and $\epsilon >0$ 
\begin{equation*}
 \mu(B_{\epsilon}(x)) \leq K \epsilon^{n(1-\frac{1}{p})}
\end{equation*}
for some $p > n$ and $K>0$.  Indeed, if $m-n \geq n$, we can take $\mu$ to be $(m-n)$-dimensional Hausdorff measure on a single level set $L_x$.  Then $\mu$ will satisfy the above condition for any $p$, but $\alpha$ will consist of a single Dirac mass. 

The preceding lemma allows use to immediately translate the regularity results of Loeper and Liu to the present setting.

\newtheorem{cont}[exist]{Corollary}
\begin{cont}
 Suppose that $Y$ is $c$-convex with connected level sets $L_x(y)$ for all $x \in X$ and $y \in Y$, and that \textbf{(A0)}, \textbf{(A1)}, \textbf{(A2)} and \textbf{(A3S)} hold.  Suppose that $f \in L^p(X)$ for some $p > \frac{n+1}{2}$.  Then the optimal map is Holder continuous with Holder exponent $\frac{\beta(n+1)}{2n^2+\beta(n-1)}$, where $\beta=1-\frac{n+1}{2p}$.
\end{cont} 

The higher regularity results of Ma, Trudinger and Wang require $C^{2}$ smoothness of the density $h$.  As the following example demonstrates, however, smoothness of $f$ does not even imply continuity of $h$.

\newtheorem{nonsmooth}[exist]{Example}
\begin{nonsmooth}
Let 
\begin{equation*}
X =\{x=(x_1,x_2): -1<x_1<1, -1<x_2<\phi(x_1)  \} \subseteq \mathbb{R}^2 
\end{equation*}
where $\phi:(-1,1) \rightarrow (-1,1)$ is a $C^{\infty}$ function such that $\phi(x_1)=0$ for all $-1<x_1<0$, $\phi(1)=1$ and $\phi$ is strictly increasing on $(0,1)$.  Let $Y =(0,1) \subseteq \mathbb{R}$ and $c(x,y)=x_2y$.    Then $Y$ is $c$-convex and $c$ satisfies \textbf{(A0)}-\textbf{(A3S)}.  The level sets $L_x$ are simply the curves $\{x:x_2 =c\}$ for constant values of $c \in (-1,1)$ and $Z=(-1,1)$.  Set $f(x) = k$, where $k$ is a constant chosen so that $\mu$ has total mass 1.  The density $h$ is then easy to compute; it is simply the length of the line segment $Q^{-1}(z)$.  For $z<0$, $h(z)=2k$; however, for $z>0$, $h(z)=k(1-\phi^{-1}(z)) <k$. \footnote{It should be noted that the while the boundary of $X$ is not smooth here, this is not the reason for the discontinuity in $h$; the corners of the boundary can be mollified and the density will still be discontinuous at $0$.}
\end{nonsmooth}

On the other hand, we should note that is possible for $\alpha$ to be smooth even when $\mu$ is singular.  This will be the case if, for example, $\mu$ is $n$-dimensional Hausdorff measure concentrated on some smooth $n$-dimensional surface $S$ which intersects the $L_x$'s transversely.

Finally, we exploit Loeper's counterexample, which shows that, when $m=n$ and \textbf{(A3W)} fails, there are smooth
densities for which the optimal map is not continuous.  

\newtheorem{noncont}[exist]{Corollary}
\begin{noncont}
 Suppose that $Y$ is $c$-convex and that the level sets $L_x(y)$ are connected for all $x \in X$ and $y \in Y$.  Assume \textbf{(A0)}, \textbf{(A1)}, and \textbf{(A2)} hold but \textbf{(A3W)} fails.  Then there are smooth marginals $\mu$ on $X$ and $\nu$ on $Y$ such that the optimal map is discontinuous.
\end{noncont}

\begin{proof}
 Using Proposition \ref{formula}, it is easy to check that $u: X \rightarrow \mathbb{R}$ is $c$-concave if and only it $u(x)=v(Q(x))+c(x,y_0)$ for some $b$-concave $v: Z \rightarrow \mathbb{R}$.  By \cite{loeper}, we know that if \textbf{(A3W)} fails, then the set of $C^1$, $b$-concave functions is not dense in the set of all $b$-concave functions in the $L^{\infty}(Z)$ topology.  From this it follows easily that the set of $C^1$, $c$-concave functions is not dense in the set of all $c$-concave functions in the $L^{\infty}(X)$ topology.  The argument in \cite{loeper} now implies the desired result.
\end{proof}

\section{Regularity for non-c-convex targets}
The counterexamples of Ma, Trudinger and Wang, combined with the results in the previous section imply that we cannot hope that the optimizer is continuous for arbitrary smooth data if the level sets $L_x(y)$ are not independent of $y$.  It is then natural to ask for which marginals \emph{can} we expect the optimal map to smooth?  In this section, we study this question in the special case when $m=2$ and $n=1$.  We identify conditions on the interaction between the marginals and the cost that allow us to find an explicit formula for the optimal map and prove that it is continuous.

We will assume $Y=(a,b) \subset \mathbb{R}$ is an open interval and that $X$ is a bounded domain in $\mathbb{R}^2$. We will also assume that $c \in C^2(\overline{X} \times \overline{Y})$ satisfies \textbf{(A2)}, which in this setting simply means that the gradient $\triangledown_x(\frac{\partial c}{\partial y})$ never vanishes.  Therefore, the level sets $L_x(y)$ will all be $C^1$ curves.  We define the following set:

\begin{equation*}
 P=\Big\{\tilde{x} \in \overline{X} : \forall\text{ } y_0<y_1 \in Y,\text{ } x \in L_{\tilde{x}}(y_0),\text{ we have } \frac{\partial c(\tilde{x},y_1)}{\partial y}  \leq \frac{\partial c(x,y_1)}{\partial y}\Big\}
\end{equation*}
  When the level sets $L_x(y)$ are independent of $y$, $P$ is the entire domain $X$.  If not, $P$ consists of points $\tilde{x}$ for which the level sets $L_{\tilde{x}}(y)$ evolve with $y$ in a monotonic way.  $L_{\tilde{x}}(y_1)$ divides the region $X$ into two subregions: $\{x:\frac{\partial c(\tilde{x},y_1)}{\partial y} > \frac{\partial c(x,y_1)}{\partial y}\}$ and $\{x:\frac{\partial c(\tilde{x},y_1)}{\partial y} \leq \frac{\partial c(x,y_1)}{\partial y}\}$. $\tilde{x} \in P$ ensures that for $y_0<y_1$, the set $L_{\tilde{x}}(y_0)$ will lie entirely in the latter region.  For interior points, the curves $L_{\tilde{x}}(y_0)$ and $L_{\tilde{x}}(y_1)$ will generically intersect transversely and so $L_{\tilde{x}}(y_0)$ will interect both of these regions; therefore, $P$ will typically consist only of boundary points.  At each boundary point $\tilde{x}$, we can heuristically view the level curves $L_{\tilde{x}}(y)$ as rotating about the point $\tilde{x}$; $P$ consists of those points which rotate in a particular fixed direction.    

In what follows, $\gamma$ will be a solution to the Kantorovich problem.  The support of $\gamma$, or $spt(\gamma)$, is the smallest closed subset of $X \times Y$ of full mass.

\newtheorem{mono}{Lemma}[section]
\begin{mono}\label{pos}
Suppose $\tilde{x} \in P, \text{} x \in X,\text{} y_0,y_1 \in Y$ and $(\tilde{x},y_1), (x,y_0) \in spt(\gamma)$.  Then $\frac{\partial c(x,y_1)}{\partial y} \leq \frac{\partial c(\tilde{x},y_1)}{\partial y}$ if $y_0 <y_1$ and $\frac{\partial c(x,y_1)}{\partial y} \geq \frac{\partial c(\tilde{x},y_1)}{\partial y}$ if $y_0 >y_1$
\end{mono}
\begin{proof}
The support of $\gamma$ is $c$-monotone (see \cite{sk} for a proof); this means that $c(\tilde{x},y_1) + c(x,y_0) \leq c(\tilde{x},y_0) + c(x,y_1)$.  If $y_0 <y_1$, this implies 
\begin{equation}\label{mono}                                                                                          \int_{y_0}^{y_1}\frac{\partial c(\tilde{x},y)}{\partial y} dy \leq \int_{y_0}^{y_1}\frac{\partial c(x,y)}{\partial y} dy.                                                                                                                   \end{equation}
Assume $\frac{\partial c(\tilde{x},y_1)}{\partial y} > \frac{\partial c(x,y_1)}{\partial y}$. We claim that this implies $\frac{\partial c(\tilde{x},y)}{\partial y} > \frac{\partial c(x,y)}{\partial y}$ for all $y \in [y_0,y_1]$, which contradicts (\ref{mono}).  To see this, suppose that there is some $y \in [y_0,y_1]$ such that $\frac{\partial c(\tilde{x},y)}{\partial y} \leq \frac{\partial c(x,y)}{\partial y}$; the Intermediate Value Theorem then implies the existence of a $\overline{y} \in [y,y_1)$ such that $\frac{\partial c(\tilde{x},\overline{y})}{\partial y} = \frac{\partial c(x,\overline{y})}{\partial y}$, or $x \in L_{\tilde{x}}(\overline{y})$.  This, together with our assumption $\frac{\partial c(\tilde{x},y_1)}{\partial y} > \frac{\partial c(x,y_1)}{\partial y}$, violates the condition $\tilde{x} \in P$.  

A similar argument shows $\frac{\partial c(\tilde{x},y_1)}{\partial y} \geq \frac{\partial c(x,y_1)}{\partial y}$ if $y_0 >y_1$.
\end{proof}

\newtheorem{splitting}[mono]{Definition}
\begin{splitting}
 We say $y$ splits the mass at $x$ if
\begin{equation*}
 \mu\Big(\{\overline{x} : \frac{\partial c(x,y)}{\partial y}  < \frac{\partial c(\overline{x},y)}{\partial y}\}\Big)=\nu( [0,y))
\end{equation*}
If $\mu$ and $\nu$ are absolutely continuous with respect to Lebesgue measure, this is equivalent to 
\begin{equation*}
 \mu\Big(\{\overline{x} : \frac{\partial c(x,y)}{\partial y}  > \frac{\partial c(\overline{x},y)}{\partial y}\}\Big)=\nu( [y,1])
\end{equation*}
\end{splitting}

Lemma \ref{pos} immediately implies the following.

\newtheorem{divi}[mono]{Lemma}
\begin{divi}\label{divi}
Suppose $\mu$ and $\nu$ are absolutely continuous with respect to Lebesgue measure.  Then if $\tilde{x} \in P, y \in Y$ and $(\tilde{x},y) \in spt(\gamma)$, $y$ splits the mass at $\tilde{x}$.  
\end{divi}

\newtheorem{isone}[mono]{Lemma}
\begin{isone} \label{ex}
 Suppose $\mu$ and $\nu$ are absolutely continuous with respect to Lebesgue.  Then, for each $x \in X$ there is a $y \in Y$ that splits the mass at $x$.
\end{isone}
\begin{proof}
The function $y \mapsto f_x(y):= \mu\Big(\{\overline{x} : \frac{\partial c(x,y)}{\partial y}  < \frac{\partial c(\overline{x},y)}{\partial y}\}\Big)-\nu\big( [0,y)\big)$ is continuous.  Observe that $f_x(0) \geq 0$ and $f_x(1) \leq 0$; the result now follows from the Intermediate Value Theorem.
\end{proof}
Similarly, it is straightforward to prove the following lemma.
\newtheorem{isoney}[mono]{Lemma}
\begin{isoney}\label{exy}
 Suppose $\mu$ and $\nu$ are absolutely continuous with respect to Lebesgue.  Then, for each $y \in Y$ there is an $x \in X$ such that $y$ splits the mass at $\overline{x}$ if and only if $\overline{X} \in L_x(y)$.
\end{isoney}

\newtheorem{mass}[mono]{Definition}
\begin{mass}
 Let $\tilde{x} \in P$.  We say $\tilde{x}$ satisfies the mass comparison property (MCP) if for all $y_0 < y_1 \in Y$ we have 
\begin{equation*}
 \mu\Big( \bigcup_{y \in [y_0, y_1]}L_{\tilde{x}}(y)\Big) < \nu\big([y_0, y_1]\big)
\end{equation*}
\end{mass}

 In the case when the level sets $L_x(y)$ are independent of $y$, the MCP is satisfied for all $x \in P=\overline{X}$ as long as $\mu$ assigns zero mass to every $L_x(y)$ and $\nu$ assigns non-zero mass to every open interval.  Alternatively, in view of the previous section, we know that in this case the cost has the form $c(Q(x),y)$, where $Q: X \rightarrow Z$ and $Z =[z_0,z_1] \subseteq \mathbb{R}$ is an interval; the MCP boils down to the assumption that $\alpha$ assigns zero mass to all singletons and $\nu$ assigns non-zero mass to every open interval. 

\newtheorem{uniqueP}[mono]{Lemma}
\begin{uniqueP}\label{uniqueP}
  Suppose $\mu$ and $\nu$ are absolutely continuous with respect to Lebesgue measure and that $\tilde{x} \in P$ satisfies the MCP.  Then there is a unique $y \in Y$ that splits the mass at $\tilde{x}$.
\end{uniqueP}
\begin{proof}
 Existence follows from Lemma \ref{ex}; we must only show uniqueness.  Suppose $y_0<y_1 \in Y$ both split the mass at $\tilde{x}$.  For any $x$ such that $\frac{\partial c(x,y_0)}{\partial y}  > \frac{\partial c(\tilde{x},y_0)}{\partial y}$ and $\frac{\partial c(x,y_1)}{\partial y}  < \frac{\partial c(\tilde{x},y_1)}{\partial y}$ the Intermediate Value Theorem yields a $y \in [y_0, y_1]$ such that $x \in L_{\tilde{x}}(y)$; hence, 
\begin{eqnarray*}
 \big\{x : \frac{\partial c(x,y_0)}{\partial y}  > \frac{\partial c(\tilde{x},y_0)}{\partial y}\big\} &\bigcap& \big\{x:\frac{\partial c(x,y_1)}{\partial y}  <\frac{\partial c(\tilde{x},y_1)}{\partial y}\big\} \\
& \subseteq  & \bigcup_{y \in [y_0, y_1]}L_{\tilde{x}}(y)
\end{eqnarray*}
Therefore
\begin{eqnarray}
 \mu\Big(\big\{x : \frac{\partial c(x,y_0)}{\partial y}  > \frac{\partial c(\tilde{x},y_0)}{\partial y}\big\} &\bigcap&  \big\{x:\frac{\partial c(x,y_1)}{\partial y}  < \frac{\partial c(\tilde{x},y_1)}{\partial y}\big\}\Big) \nonumber \\
& \leq  & \mu\big( \bigcup_{y \in [y_0, y_1]}L_{\tilde{x}}(y)\big) \label{mcv}
\end{eqnarray}
Now, absolute continuity of $\mu$ and $\nu$ together with the assumption that $y_0$ and $y_1$ split the mass at $\tilde{x}$ yield
\begin{eqnarray}
 \mu\Big(\big\{x : \frac{\partial c(x,y_0)}{\partial y}  > \frac{\partial c(\tilde{x},y_0)}{\partial y}\big\} &\bigcap& \big\{x:\frac{\partial c(\tilde{x},y_1)}{\partial y}  < \frac{\partial c(\tilde{x},y_1)}{\partial y}\big\}\Big) \nonumber  \\
&= & \nu\big([y_0,y_1]\big) \label{meas}
\end{eqnarray}
 Combining (\ref{mcv}) and (\ref{meas}) and the MCP now yields a contradiction.
\end{proof}
We are now ready to prove the main result of this section.
\newtheorem{form}[mono]{Theorem}
\begin{form}\label{form}
 Suppose $\mu$ and $\nu$ are absolutely continuous with respect to Lebesgue.  Suppose that for all $x,y \in \overline{X} \times \overline{Y}$ such that $y$ splits the mass at $x$ there exists an $\tilde{x} \in P \cap L_x(y)$ satisfying the MCP.  Then for each $x \in \overline{X}$ there is a unique $y \in \overline{Y}$ that splits the mass at $x$.  Moreover, $(x,y) \in spt(\gamma)$ and $(x, \overline{y}) \notin spt(\gamma)$ for all other $\overline{y} \in \overline{Y}$.  Therefore, the optimal map is well defined everywhere.
\end{form}
\begin{proof}
 For each $x \in X$, by Lemma \ref{ex} we can choose $y \in Y$ that splits the mass at $x$; the hypothesis then implies the existence of $\tilde{x} \in P \cap L_x(y)$ satisfying the MCP.  Lemmas \ref{uniqueP} and \ref{divi} imply that $(\tilde{x},y) \in spt(\gamma)$.  

We now show that 
\begin{equation}\label{claim}
(x, y') \notin spt(\gamma) \text{ for all }y' \neq y.
\end{equation}
The proof is by contradiction; to this end, assume $(x,y') \in spt(\gamma)$ for some $y' \neq y$.  Suppose $y' > y$; choose $\overline{y} \in (y, y')$.  By Lemma \ref{exy}, we can choose $\overline{x}$ such that $\overline{y}$ splits the mass at $\overline{x}$.  Now use the hypothesis of the theorem again to find $\tilde{\tilde{x}} \in P \cap L_{\overline{x}}(\overline{y})$ satisfying the MCP and note that $(\tilde{\tilde{x}},\overline{y}) \in spt(\gamma)$.  By Lemma \ref{uniqueP}, $\tilde{x} \notin L_{\tilde{\tilde{x}}}(\overline{y})$, and so Lemma \ref{pos} implies $\frac{\partial c(\tilde{x},\overline{y})}{\partial y}  < \frac{\partial c(\tilde{\tilde{x}},\overline{y})}{\partial y}$.

Therefore, 
\begin{eqnarray*}
\frac{\partial c(x,\overline{y})}{\partial y}  &\leq& \frac{\partial c(\tilde{x},\overline{y})}{\partial y} \\
& < &\frac{\partial c(\tilde{\tilde{x}},\overline{y})}{\partial y}
\end{eqnarray*}
But now $(x, y'), (\tilde{\tilde{x}},\overline{y}) \in spt(\gamma)$ and $y' > \overline{y}$ contradicts Lemma \ref{pos}.  An analogous argument implies that we cannot have $(x,y') \in spt(\gamma)$ for $y' <y$, completing the proof of (\ref{claim}). 

Now, note that we must have $(x,\overline{y}) \in spt(\gamma)$ for \textit{some} $\overline{y} \in Y$ and so the preceding argument implies $(x,y) \in spt(\gamma)$.

Finally, we must show that there is no other $y' \in Y$ which splits the mass at $x$; this follows immediately, as if there were such a $y'$, an argument analogous to the preceding one would imply that $(x,y') \in spt(\gamma)$, contradicting (\ref{claim}).

\end{proof}
Note that we can use Theorem \ref{form} to derive a formula for the optimal map:
\begin{equation*}
 F(x) := \sup_y\Big\{y: \mu\Big(\{\overline{x} : \frac{\partial c(x,y)}{\partial y}  < \frac{\partial c(\overline{x},y)}{\partial y}\}\Big)>\nu( [0,y))\Big\}
\end{equation*}

\newtheorem{reg}[mono]{Corollary}
\begin{reg}
 Under the assumptions of the preceding theorem, the optimal map is continuous on $\overline{X}$.  
\end{reg}

\begin{proof}
Choose $x_k \rightarrow x \in \overline{X}$ and set $y_k =F(x_k)$; we need to show $y_k \rightarrow F(x)$.  Set $\overline{y}=\limsup_{k \rightarrow \infty} y_k \in \overline{Y}$; by passing to a subsequence we can assume $y_k \rightarrow \overline{y}$.  As $spt(\gamma)$ is closed by definition, we must have $(x, \overline{y}) \in spt(\gamma)$ and so Theorem \ref{form} implies $\overline{y} =F(x)$.  A similar argument implies $\liminf_{k \rightarrow \infty} y_k =F(x)$, completing the proof.
\end{proof}

The following example illustrates the implications of the preceding Corollary. 

\newtheorem{ex}[mono]{Example}
\begin{ex}
 Let $X$ be the quarter disk:
\begin{equation*}
 X=\big\{(x_1,x_2) : x_1 > 0, x_2 > 0, x_1^2+x_2^2 < 1\big\}
\end{equation*}
Let $Y =(0, \frac{\pi}{2})$ and take $\mu$ and $\nu$ to be uniform measures on $X$ and $Y$, respectively, scaled so that both have total mass 1.  Let $c(x,y)=-x_1\cos(y)-x_2\sin(y)$; this is equivalent to the Euclidean distance between $x$ and the point on the unit circle parametrized by the polar angle $y$.  We claim that the optimal map takes the form $F(x)=\arctan(\frac{x_2}{x_1})$; that is, each point $x$ is mapped to the point $\frac{x}{|x|}$ on the unit circle.  Indeed, note that 
\begin{equation}\label{spt}
 c(x,y) \geq -\sqrt{x_1^2+x_2^2}
\end{equation}
with equality if and only if $y=F(x)$, and that uniform measure on the graph $(x,F(x))$ projects to $\mu$ and $\nu$, implying the desired result.  Now observe that $F$ is discontinuous at $(0,0)$; in fact, $((0,0), y)$ satisfies (\ref{spt}) for all $y \in Y$ so the optimal measure pairs the origin with every point.  Note that the conditions of Theorem \ref{form} fail in this case, as every $y \in Y$ splits the mass at $(0,0) \in \overline{X}$.  

Now suppose instead that $\nu$ is uniform measure on $[0, \frac{\pi}{4}]$, rescaled to have total mass $1$.  It is not hard to check that $(0,x_2)$ is in  $P$ and satisfies the MCP for all $x_2$.  Now, for all $(x,y) \in Y$ such that $y$ splits the mass at $x$, it is straightforward to verify that we have some $(0,x_2) \in L_x(y)$; hence, Corollary \ref{spt} implies continuity of the optimizer.
\end{ex}

\end{document}